\documentclass[11pt]{article}

\usepackage{amsfonts}
\usepackage{amssymb,amsmath,amsthm, enumerate}
\usepackage{latexsym}
\usepackage{fullpage}
\usepackage{hyperref}

\newtheorem{theorem}{Theorem}[section]

\newtheorem{prop}[theorem]{Proposition}

\newtheorem{lemma}[theorem]{Lemma}
\newtheorem{cor}[theorem]{Corollary}

\newtheorem{claim}[theorem]{Claim}

\theoremstyle{definition}

\newcounter{tenumerate}

\def\P{\mathbb{P}}

\newcommand{\one}{\1}

\newcommand{\reff}{R_{\mathrm{eff}}}

\renewcommand{\epsilon}{\varepsilon}

\newcommand{\1}{\mathbf{1}}

\DeclareMathOperator{\var}{Var}

\newcommand{\N}{{\mathbb N}}

\newcommand{\E}{{\mathbb E}}
\newcommand{\remove}[1]{}

\renewcommand{\leq}{\leqslant}
\renewcommand{\geq}{\geqslant}

\newcommand{\cov}{\mathrm{Cov}}

\def\XXint#1#2#3{{\setbox0=\hbox{$#1{#2#3}{\int}$}
\vcenter{\hbox{$#2#3$}}\kern-.5\wd0}}

\begin{document}

\title{{\bf On cover times for 2D lattices}}

\author{Jian Ding \thanks{Part of the work was carried out while the author was supported partially by Microsoft Research.} \\ Stanford University}

\date{}
\maketitle

\begin{abstract}
We study the cover time $\tau_{\mathrm{cov}}$ by (continuous-time) random walk on the 2D box of side length $n$ with wired boundary or on the 2D torus,
and show that in both cases with  probability approaching $1$ as $n$ increases,
$\sqrt{\tau_{\mathrm{cov}}}=\sqrt{2n^2}[\sqrt{2/\pi} \log n + O(\log\log n)]$. This improves a result of Dembo, Peres, Rosen, and Zeitouni (2004) and makes progress
towards a conjecture of Bramson and Zeitouni (2009).
\end{abstract}

\section{Introduction}
We consider the random walk on a 2D box/torus and study a fundamental parameter, the cover time $\tau_{\mathrm{cov}}$, which is the first time when the random walk has visited
every single vertex of the underlying graph.

Let $A \subset \mathbb{Z}^2$ be a 2D box and denote by $\partial A = \{v\in A: \exists u\in \mathbb{Z}^2 \setminus A \mbox{ such that } u\sim v\}$ the boundary set of $A$. We say a random walk on $A$ with wired boundary
if we identify $\partial A$ as a single vertex and run the random walk on the wired graph. Formally, the transition kernel of the random walk is given by
$$p(u, v) = \begin{cases}
\tfrac{1}{4} \one_{v\sim u}\,, &\mbox{ if } u\in A \setminus \partial A\,,\\
\tfrac{d_{v, \partial A}}{\sum_{w\in A\setminus \partial A} d_{w, \partial A}} \one_{v\in A \setminus \partial A}\,, &\mbox{ if } u\in \partial A\,,
\end{cases}$$
where $d_{v, \partial A} = |\{v'\in \partial A: v'\sim v\}|$. Throughout this work, we consider continuous-time random walk, where the random walk makes jumps according to a Poisson clock with rate 1 and the jumping rule follows the transition kernel (of the corresponding discrete walk).

We give the following estimate on $\tau_{\mathrm{cov}}$ for a random walk on 2D box of side length $n$ with wired boundary condition. Throughout the paper, the notation ``with high probability'' means that with probability approaching 1 as $n\to \infty$.
\begin{theorem}\label{thm-dirichlet}
The cover time $\tau_{\mathrm{cov}}$ for a random walk in an $n\times
n$ 2D box with wired boundary with high probability satisfies the following for an absolute constant $c>0$,
$$\sqrt{2/\pi} \log n - c\log\log n \leq  \sqrt{\tau_{\mathrm{cov}}/2n^2} \leq  \sqrt{2/\pi} \log n + \log\log n\,.$$
\end{theorem}
Denote by $\mathbb{Z}_n^2$ a 2D torus with total number of vertices $n^2$. We prove an analogous result to the preceding theorem for the random walk $\mathbb{Z}_n^2$.
\begin{theorem}\label{thm-free}
The cover time $\tau_{\mathrm{cov}}$ for a random walk on $\mathbb{Z}_n^2$ with high probability satisfies the following for absolute constants $c, C>0$,
$$\sqrt{2/\pi} \log n - c\log\log n \leq  \sqrt{\tau_{\mathrm{cov}}/2n^2} \leq  \sqrt{2/\pi} \log n + C\log\log n\,.$$
\end{theorem}

\noindent{\bf Remarks.} (1) Our results extend to discrete-time random walk as the number of jumps made up to time $t$ is highly concentrated with a standard deviation $\sqrt{t}$. (2) It is clear from our proof that the expected cover time $t_{\mathrm{cov}}$ satisfies that $\sqrt{t_{\mathrm{cov}}/2n^2} = \sqrt{2/\pi} \log n + O(\log\log n)$. (3) By analogy with regular trees, we believe that $\log\log n$ is the correct order for the second term of the normalized cover time. (4) The upper bound in Theorem~\ref{thm-dirichlet} holds if we replace $\log \log n$ by an arbitrary sequence $a_n$ with $a_n \to \infty$ (as indicated in the proof). Since we believe the true value of the cover time should be around $\sqrt{2/\pi} \log n - c^*\log \log n$ for a positive constant $c^*$, we decide such potential improvement is irrelevant. (5) The wired boundary corresponds to Dirichlet
boundary condition for 2D Gaussian free field, and the torus case corresponds to the GFF on a torus. Among other things, one connection is that the covariances for the GFF are given by the Green function (see \eqref{eq-def-green-function}) of the corresponding random walk.

Our estimates improve a result of Dembo, Peres, Rosen, and Zeitouni \cite{DPRZ04}, in which they show that with high probability $\sqrt{\tau_{\mathrm{cov}}/2n^2} = (\sqrt{2/\pi} + o(1)) \log n$ for $\mathbb{Z}_n^2$. The question on the limiting law
for $\sqrt{\tau_{\mathrm{cov}}/n^2}$ was initialized in \cite{DPRZ04}, and it was explicitly conjectured by Bramson and Zeitouni \cite{BZ09} that $\sqrt{\tau_{\mathrm{cov}}/n^2}$ is tight after centering by its median. Our results can be
seen as progress toward this conjecture.

The cover time on a graph, turns out to have an intimate connection with the maximum of the discrete Gaussian free field (GFF) on the graph.
The GFF $\{\eta_v: v\in A\}$ on 2D box with Dirichlet boundary is a mean zero Gaussian process which takes value 0 on $\partial A$
and satisfies the following Markov field condition for all $v\in A\setminus \partial A$: $\eta_v$ is distributed as a Gaussian variable with variance $1/4$ and mean the average over the neighbors (in graph distance) given the GFF on $A\setminus \{v\}$ (see also Section~\ref{sec:prelim}). Throughout the paper,
we use the notation
\begin{equation}\label{eq-def-M-n}
M_n = \sup_{v\in A} \eta_v\,, \mbox{ and } m_n = \E M_n\,.
\end{equation}
where $A$ is a 2D box of side length $n$ and $\eta_v$ is the GFF on $A$ with Dirichlet boundary condition. We compare our cover time results with the following result on
the tightness of the maximum of the GFF on 2D box due to Bramson and Zeitouni \cite{BZ10}, which serves as a fundamental ingredient in our proof (note that definitions of the GFF in our paper and in \cite{BZ10}
have different normalization -- by a factor of 2).
\begin{theorem}\label{thm-BZ}\cite{BZ10}
The sequence of random variables $M_n - m_n$ is tight and $$ m_n  =  \sqrt{2/\pi} \big(\log n - \tfrac{3}{8\log 2} \log\log n\big) + O(1)\,.$$
\end{theorem}
We remark that in view of the analogy to the GFF, it is not at all obvious why the cover time seems to exhibit the same behavior on a box with wired boundary and a torus, as the maximum
for the GFF does have different deviation in two cases. In order to see that the maximum for GFF on 2D torus (with a fixed vertex being 0) has deviation of order $\sqrt{\log n}$, we take a box of side-length $n/2$ in side the torus and argue that (i) the
average (or a suitable weighted average) for the GFF over the boundary of the small box is a mean zero Gaussian variable with variance of order $\log n$; (ii) the deviation of the average of the GFF over this boundary will propagate to the maximum,
since the GFF of every vertex in a slightly smaller and centered box has (roughly) the average as the mean conditioning on the GFF over the boundary.

\subsection{Related works}
In a work of Ding, Lee, and Peres \cite{DLP10}, a useful connection between cover
times and GFFs had been demonstrated by showing that, for any graph, the cover time is
equivalent, up to a universal multiplicative constant, to the
product of the number of edges and the square of the expected supremum for the GFF. This connection was recently strengthened by \cite{Ding11}, which
obtained the leading order asymptotics
of the cover time via Gaussian free
fields in bounded-degree graphs as well as general trees (together with an
exponential concentration around its mean for the case of trees). We will use some of the ideas therein.

Before the work \cite{DLP10}, the connection between cover times and GFFs for certain specific graphs, has already been a folklore. This was highlighted through the
analogy in 2D lattice. Bolthausen, Deuschel and Giacomin \cite{BDG01} established the asymptotics for the maximum of GFF on 2D box, by exploring a certain tree structure in the 2D lattice;
in \cite{DPRZ04}, the asymptotics for the cover time was calculated through an analogous tree structure but the proof was significantly more involved. A similar tree structure also appeared in
the work of Daviaud \cite{Daviaud06} who studied the extremes for the GFF. The key idea of \cite{BZ10} is to construct a tree structure of this type but with conceptual novelty.

For regular trees, Ding and Zeitouni \cite{DZ11} computed precisely the second order term for the cover time, improving the asymptotics result of Aldous \cite{Aldous91} and complementing the tightness result of
\cite{BZ09}.

\subsection{A brief discussion}

We remark that it was highly nontrivial for \cite{DPRZ04} to explore a tree structure and succeed to demonstrate the asymptotics for the cover time in 2D torus. However, it seems unlikely to obtain better estimate using the method employed therein.
A natural attempt would be to use a modified tree-structure (the so-called Modified branching random walk) as in \cite{BZ10} to obtain a precise estimate for the cover time. Among other things, there appear to be two significant challenges when trying to implement such strategy: For one thing, the structure in \cite{BZ10} comes
arise from a random partitioning of 2D box (alternatively, a smoothed branching Brownian random walk) and it is not clear how such a structure shall apply to random walk. For another,
the linear structure for Gaussian process great simplifies the issue of controlling the correlations. In particular, the well-known comparison theorems for Gaussian processes (see \cite{slepian62, Fernique75})
allowed \cite{BZ10} to conveniently switch between several modifications of GFFs/branching random walks. A lacking of such comparison theorems for cover times raises a conceptual challenge.
The proof in our work, get around these two issues by explicitly using a two-level structure, which in a sense amounts to bury (use) a tree structure as of \cite{BZ10} for GFF instead of random walk --- this is validated by the connection between the random walk and the GFF.
Finally, as demonstrated in \cite{DZ11} there is a difference of order $\log \log n$ between the normalized cover time and the maximum of GFF on a binary tree of $n$ vertices. Such a difference should presumably also exhibit in 2D lattices, which suggests an obstacle for attempts to improve our estimates on the cover time using its connection with the GFF.

\subsection{Preliminaries}\label{sec:prelim}

\noindent{\bf Electric networks.} In this work, we focus on random
walks on 2D lattices and thus have no attempt to touch the general
setting for a random walk on network. Nevertheless, we occasionally
use the point of view of electric network, which is obtained by
placing a unit resistor on each edge of the graph. Associated to
such an electrical network are the classical quantity $\reff : V
\times V \to \mathbb R_{\geq 0}$ which is referred to as the {\em
effective resistance} between pairs of nodes. We refer to \cite[Ch.
9]{LPW09} for a discussion about the connection between electrical
networks and the corresponding random walk (where the random walker started at $u$ moves to a vertex $v$ with probability proportional to edge conductance on $(u, v)$).


For convenience, we will mainly work with {\em continuous-time}
random walks where jumps are made according to the discrete-time
random walk and the times spent between jumps are i.i.d.\
exponential random variables with mean 1. See \cite[Ch. 2]{AF} for
background and relevant definitions. We remark that our results
automatically extend to discrete time random walk. Note the number of steps $N(t)$ performed by a
continuous-time random walk up to time $t$, has Poisson distribution
with mean $t$. Therefore, $N(t)$ exhibits a Gaussian type
concentration around $t$ with standard deviation bounded by
$\sqrt{t}$. This implies that the concentration result in
Theorems~\ref{thm-dirichlet} and \ref{thm-free} hold for
discrete-time case.

\smallskip

\noindent{\bf Gaussian free field.} Consider a connected graph $G =
(V, E)$. For $U \subset V$, the Green function $G_U(\cdot, \cdot)$ of the
discrete Laplacian is given by
\begin{equation}\label{eq-def-green-function}
G_U(x, y) = \E_x(\mbox{$\sum_{k=1}^{\tau_U - 1}$} \one\{S_k =
y\})\,, \mbox{ for all } x, y\in V\,,
\end{equation}
where $\tau_U$ is the hitting time to set $U$ for random walk
$(S_k)$, defined by (the notation applies throughout the paper)
\begin{equation}\label{eq-def-tau-A}
\tau_U = \min\{k\geq 0: S_k \in U\}\,.
\end{equation}
The discrete Gaussian free field (GFF) $\{\eta_v: v\in V\}$ with
Dirichlet boundary on $U$ is then defined to be a mean zero Gaussian
process indexed by $V$ such that the covariance matrix is given by
the normalized Green function $(G_U(x, y)/ d_y)_{x, y\in V}$, where $d_y$ is the degree of vertex $y$. It is
clear to see that $\eta_v = 0$ for all $v\in U$. For more preliminary background on Gaussian free field, we encourage the reader to refer to \cite{Janson97, MR06} for a good account.

\smallskip

\noindent {\bf Cover times and local times.}
For a vertex $v \in V$
and time $t$, we define the {\em local time $L_t^v$} by
\begin{equation}\label{eq:localtimedef}
L_t^v = \frac{1}{d_v} \int_0^t \1_{\{X_s = v\}} ds\,.
\end{equation}
It is obvious that local times are crucial in the study of cover
times, since $$\tau_{\mathrm{cov}} = \inf\{t>0: L^v_t >0 \mbox{ for
all } v\in V\}\,.$$ For that purpose, it turns out that it is
convenient to decompose the random walk into excursions at the
origin $v_0 \in V$. This motivates the following definition of the
inverse local time $\tau(t)$:
\begin{equation}\label{eq:inverselt}
\tau(t) = \inf\{s: L_s^{v_0} > t\}.
\end{equation}
We study the cover time via analyzing the local time process
$\{L^v_{\tau(t)}: v\in V\}$. In this way, we measure the cover time
in terms of $\tau(t)$ and we are indeed working with the {\em cover
and return time}, defined as
\begin{equation}\label{eq:covandreturn}
\tau_{\mathrm{cov}}^{\circlearrowright} = \inf \left\{ t >
\tau_{\mathrm{cov}} : X_t = X_0 \right\}.
\end{equation}
In the asymptotic and concentration sense considered in this work,
the difference between cover times $\tau_{\mathrm{cov}}$ and
$\tau_\mathrm{cov}^{\circlearrowright}$ is negligible. In order to
see this, define $t_{\mathrm{hit}} = \max_{u, v}\E_u \tau_v$ (where $\tau_v = \tau_{\{v\}}$ is defined as in \eqref{eq-def-tau-A})and  note that
$$\tau_\mathrm{cov}^{\circlearrowright} \preceq \tau_{\mathrm{cov}} + \tau_{\mathrm{hit}}\,,$$
where $\preceq$ means stochastic domination, and
$\tau_{\mathrm{hit}}$ measures the time it takes the random walk to
goes back to the origin after $\tau_{\mathrm{cov}}$. Using the recursion that
$$\max_{u, v}\P_v(\tau_u \geq 2k t_{\mathrm{hit}})  = \max_{u, v} (\P_v(\tau_u \geq  2k t_{\mathrm{hit}} \mid \tau_u \geq 2t_{\mathrm{hit}})\P_v(\tau_v \geq 2 t_{\mathrm{hit}})) \leq \tfrac{1}{2} \max_{u, v} \P_v(\tau_u \geq  2(k-1) t_{\mathrm{hit}})\,,$$
we see that $\P(\tau_{\mathrm{hit}} \geq 2k t_{\mathrm{hit}}) \leq (1/2)^k$, and hence $\tau_{\mathrm{hit}}$ is negligible.

\smallskip

\noindent{\bf Dynkin Isomorphism theory.} The distribution of the
local times for a Borel right process can be fully characterized by
a certain associated Gaussian processes; results of this flavor go
by the name of \emph{Dynkin Isomorphism theory}. Several versions
have been developed by Dynkin \cite{Dynkin83, Dynkin84}, Marcus and Rosen \cite{MR92,
MR01}, Eisenbaum \cite{Eisenbaum95} and Eisenbaum, Kaspi, Marcus,
Rosen and Shi \cite{EKMRS00}. In what follows, we present the second
Ray-Knight theorem in the special case of a continuous-time random
walk. It first appeared in \cite{EKMRS00}; see also Theorem 8.2.2 of
the book by Marcus and Rosen \cite{MR06} (which contains a wealth of
information on the connection between local times and Gaussian
processes). It is easy to verify that the continuous-time random
walk on a connected graph is indeed a recurrent strongly symmetric
Borel right process. Furthermore, in the case of random walk, the
associated Gaussian process turns out to be the GFF on the
underlying network.

\begin{theorem}[generalized Second Ray-Knight Isomorphism Theorem] \label{thm:rayknight}
Consider a continuous-time random walk on graph $G = (V, E)$ with
$v_0\in V$. Let $\tau(t)$ be defined as in \eqref{eq:inverselt}.
Denote by $\eta = \{\eta_x: x\in V\}$ the GFF on $G$ with the choice $U = \{v_0\}$ when defining Green functions as in \eqref{eq-def-green-function}. Let $P_{v_0}$ and $P_{\eta}$ be the laws of the
random walk and the GFF, respectively. Then under the measure
$P_{v_0} \times P_\eta$, for any $t > 0$
\begin{equation}\label{eq:law}
\left\{L_{\tau(t)}^x + \frac{1}{2} \eta_x^2: x\in V\right\}
\stackrel{law}{=} \left\{\frac{1}{2} (\eta_x + \sqrt{2t})^2 : x\in V
\right\}\,.
\end{equation}
\end{theorem}

\section{Random walk on 2D box with wired boundary condition}

In this section, we study the random walk on 2D box with wired boundary and prove Theorem~\ref{thm-dirichlet}.
The main work goes to the proof for the lower bound, for which we employ the sprinkling method that was used in \cite{Ding11}. As we wish to
show a fairly sharp bound on the cover time, we can only afford to decrease the time from $t$ to $(1 - 1/\log n)t$ in the sprinkling stage, such that for each ``trial'' there is
merely a chance of $\mathrm{Poly}((\log n)^{-1})$ to detect an uncovered point. To fight with such a slight probability, we need to have $\mathrm{Poly}(\log n)$ number of candidates
for trial. Apart from other things, this gives arise to the issue of controlling correlations of all these trials. We make use of the strong boundary condition in the wired case (which decouples the random walk), together with carefully
chosen candidates, to solve this issue.

\subsection{Concentration for inverse local time} In this work, we typically measure
the cover time $\tau_{\mathrm{cov}}$ by the inverse local time
$\tau(t)$. This is an efficient measurement only if $\tau(t)$ is
highly concentrated, which we show in this subsection.

\begin{lemma}\label{lem-concentration-tau-t}
For an $n\times n$ 2D box $A$, consider a random walk on $A$ with
wired boundary (where we identify $\partial A$ as vertex $v_0$). Let $\tau(t)$ be defined as in
\eqref{eq:inverselt}, for $t>0$. Then, the following holds uniformly for any $\lambda > 0$ and $t\geq 1$
$$\P(|\tau(t) - 4t n^2| \geq \lambda n^2 \sqrt{t}) \leq O(1/\lambda^2)\,.$$
\end{lemma}

To prove the preceding lemma, we need the next simple
claim on the joint Gaussian variables.

\begin{claim}\label{claim-joint-Gaussian}
Let $(X, Y)$ be joint mean 0 Gaussian variables such that $\cov(X, Y) = \rho$. Then
$$\cov(X^2, Y^2) = 2\rho^2\,.$$
\end{claim}
\begin{proof}
Let $\sigma_1^2 = \var X $ and $\sigma_2^2  =\var Y$.
It is clear that we can write $X = \sigma_1 Z_1$ and $Y =
\frac{\rho}{\sigma_1}Z_1 + \sqrt{\sigma_2^2 -
(\rho/\sigma_1)^2}Z_2$, where $Z_1$ and $Z_2$ are independent
standard Gaussian variables. We can now use standard moments
estimates for Gaussian variables, and obtain that
\begin{equation*}
\cov(X^2, Y^2) = \E(X^2 Y^2) - \sigma_1^2\sigma_2^2 =
\rho^2\E(Z_1^4) + (\sigma_1^2 \sigma_2^2 - \rho^2) - \sigma_1^2
\sigma_2^2 = 2\rho^2\,. \qedhere
\end{equation*}
\end{proof}

Define the \emph{Poisson kernel} $a(x, y) = a(x-y) = \sum_{n=0}^\infty (\P_0(S_n = 0) - \P_0(S_n = x-y))$ where $S_n$ is a simple random walk on $\mathbb{Z}^2$ started at the origin 0. We will need the following standard estimates on Green
functions for random walks in 2D lattices (in terms of the Poisson kernel). See, e.g., \cite[Prop.
4.6.2, Thm. 4.4.4]{LL10} for a reference.

\begin{lemma}\label{lem-green-function}
For $A\subset \mathbb{Z}^2$, consider a random walk $(S_t)$ on
$\mathbb{Z}^2$ and define $\tau_{\partial A} = \min\{j\geq 0: S_j
\in
\partial A\}$ be the hitting time to $\partial A$. For $x, y\in A$, let $G_{\partial A}(x, y)$ be the Green function as in \eqref{eq-def-green-function}.
Then
$$G_{\partial A}(x, y) = \E_x(a(S_{\tau_{\partial A}}, y)) - a(x,
y)\,.$$
Furthermore, $a(x, y) = \frac{2}{\pi} \log |x-y| + \frac{2\gamma+ \log
8}{\pi} + O(|x-y|^{-2})$, where $\gamma$ is Euler's constant.
\end{lemma}

\begin{proof}[\emph{\textbf{Proof of Lemma~\ref{lem-concentration-tau-t}}}]
For all $x\in A \setminus
\partial A$, define
$$Z_x = \eta_x^2 - \E \eta_x^2\,, \mbox{ and } Z = \sum_{x\in A\setminus \partial A} Z_x\,.$$
By Claim~\ref{claim-joint-Gaussian} and definition of Gaussian free
field, we get that for all $x, y\in A \setminus \partial A$,
$$\cov(Z_x, Z_y) =  G_{\partial A}(x, y)^2/8\,.$$
Combined with Lemma~\ref{lem-green-function}, we have that
$$\cov(Z_x, Z_y) \leq  \left(\tfrac{2}{\pi}(\log n - \log |x-y|) + O(1)\right)^2\,.$$
It follows that
\begin{align}\label{eq-var-Z}
\var Z = \sum_{x, y}\cov(Z_x, Z_y) = O(1) n^2 \sum_{k=1}^{2n}
k\left((\log n - \log k) + O(1)\right)^2  = O(n^4)\,.
\end{align}
In addition, we can estimate that
\begin{align}\label{eq-sum-gff}
\var(\mbox{$\sum_x$}\eta_x) = \sum_{x, y} \cov(\eta_x, \eta_y) =
\frac{1}{4}\sum_{x, y}G_{\partial A}(x, y)  = \frac{1}{4}\sum_x \E_x(\tau_{\partial A}) =
O(n^4)\,.
\end{align}

Recall that $v_0$ is the vertex obtained from identifying the boundary
$\partial A$. Notice that $$\tau(t) = td_{v_0} + \sum_{v\neq v_0}
d_v L^v_{\tau(t)} = 4nt + 4\sum_{v\neq v_0} L^v_{\tau(t)}\,.$$ An
application of Theorem~\ref{thm:rayknight} gives that
$$\tau(t) + Z \stackrel{law}{=} 2t|E| + 2Z + \sqrt{2t} 4 \sum_{v} \eta_v\,,$$
where $|E|$ counts for the number of edges in the wired graph and so $|E| = 2n^2 - 2n$. It follows that
$$\P(|\tau(t) - 2t|E||\geq \lambda |E|\sqrt{t}) \leq 2\P(|Z|\geq \lambda|E|\sqrt{t}/6) + \P(|\mbox{$\sum_v$} \eta_v| \geq \lambda |E|/10) \leq O(1/\lambda^2)\,,$$
where the last transition follows from an application of Chebyshev's
inequality together with estimates \eqref{eq-var-Z} and
\eqref{eq-sum-gff}.
\end{proof}

\subsection{Proof for the upper bound}\label{sec:upper-dirichlet}
We prove in this subsection the upper bound for the cover time. Note that an upper bound on the expected cover time $t_{\mathrm{cov}}$ can be obtained by an
application of Matthews method \cite{Matthews88}, as illustrated in \cite[Ch.7: Cor.25]{AF}.
The point of our proof below, which follows from an application of a union bound, is to give an upper bound on the deviation of $\tau_{\mathrm{cov}}$.

Set $t_\lambda = \frac{1}{\pi}(\log n + \lambda)^2$ for $\lambda > 0
$. By Lemma~\ref{lem-green-function},
\begin{equation}\label{eq-resistance-bound}
R_{\mathrm{eff}}(x,
\partial A) = G_{\partial A}(x, x)/4 \leq \tfrac{1}{2\pi} \log n +
O(1)\,, \mbox{ for all } x\in A\,.\end{equation}
The marginal distribution for local times satisfies that
\begin{equation}\label{eq-fact}
L^x_{\tau(t_\lambda)} \sim \sum_{i=1}^N X_i\,,
\end{equation} where $N$ is a
Poisson variable with mean $t_\lambda/R_{\mathrm{eff}}(x, \partial A)$ and $X_i$
are i.i.d.\ exponential variables with mean $R_{\mathrm{eff}}(x,
\partial A)$. This follows from the following several observations.
\begin{itemize}
\item The number of excursions at $v_0$ ($\partial A$) that occur up to $\tau(t_\lambda)$ follows a Poisson distribution with mean $d_{v_0} t_\lambda$.
\item Each excursion at $v_0$ has probability $1/d_{v_0}R_{\mathrm{eff}}(x, \partial A)$ to hit vertex $x$, independently (see, e.g., \cite[Eq. (2.5)]{LP}).
\item Starting from $x$, the local time accumulated at $x$ before hitting $v_0$ is an exponential variable with mean $R_{\mathrm{eff}}(x, \partial A)$.
\end{itemize}

By  \eqref{eq-resistance-bound} and \eqref{eq-fact}, we have that for all $x\in A$,
$$\P(L^x_{\tau(t_\lambda)} = 0)\leq \exp(-t_\lambda/R_{\mathrm{eff}}(x,
\partial A)) \leq n^{-2} O(\mathrm{e}^{-\lambda})\,.$$ A simple union bound then
gives that
$$\P(\tau_{\mathrm{cov}} \geq \tau(t_\lambda)) \leq \sum_x \P(L^x_{\tau(t_\lambda)} = 0) = O(\mathrm{e}^{-\lambda})\,.$$
Together with Lemma~\ref{lem-concentration-tau-t}, we conclude that
$$\P(\tau_{\mathrm{cov}} \geq \tfrac{4}{\pi} n^2(\log n)^2 + \lambda n^2 \log n)  = O(1/\lambda^2)\,,$$
completing the proof of the upper bound.

\subsection{Proof for the lower bound}

{\bf A packing of 2D box.} We carefully define a packing for a 2D box as follows. Let $\kappa>0$ be a constant to be
specified later. For an $n\times n$ 2D box $A$ with left bottom
corner $o = (0, 0)$, consider a collection of boxes $\mathcal{B} =
\{B_i: 0\leq i\leq m\}$ with  $m =
\lfloor\frac{(\log n)^{\kappa/3}}{12}\rfloor-1$ such that for every $i\in [m]$:
\begin{itemize}
\item The side length of each box $B_i$ satisfies $L = \frac{n}{(\log n)^{2\kappa}}$.
\item The left bottom corner of $B_i$ is $v_i = (\frac{3 i n}{(\log n)^\kappa} , n/(\log n)^{2\kappa})$.
\end{itemize}
We give a high level explanation for the choices of the parameters for the packing.
\begin{enumerate}[(a)]
\item The side length $L$ guarantees that $|m_n - m_L| = O(\log \log n)$ (this is the main reason for an error term $O(\log \log n)$ in Theorem~\ref{thm-dirichlet}).
\item The distance between each box $B_i$ and the boundary $\partial A$ is of distance $n \mathrm{Poly}((\log n)^{-1})$, which implies that the number of excursions that visited each box is $\mathrm{Poly}(\log n)$.
\item The distances between the boxes are significantly larger than the distance from boxes to the boundary, such that with high probability the random walk starting from one box would hit the boundary first before hitting other boxes.
\end{enumerate}
Property (a) bounds the magnitude of the slackness; properties (b) and (c) control the correlation for the random walks on different boxes $B_i$; the choice for the number $m$
ensures that we have enough number of trials in the sprinkling stage.

In what follows, we will abuse the notation: we denote by $\mathcal{B}$ both the collection of boxes $\{B_i\}$ and the union of boxes $\{B_i\}$, whose meaning should be clear from the context.
Following this rule, we denote by $\partial \mathcal{B} = \cup_{B\in \mathcal{B}} \partial B$ the union of the
boundaries over all boxes in $\mathcal{B}$.

The following lemma implies that our boxes in $\mathcal{B}$ are well-separated.

\begin{lemma}\label{lem-well-separated-packing}
For an $n\times n$ box $A$, define $\mathcal{B}$ as above. Then, for
any $B\in \mathcal{B}$ and $v\in B$, we have
$$\P_v(\tau_{\partial A} > \tau_{\partial \mathcal B \setminus \partial B}) \leq C(\log n)^{-\kappa/2}\,,$$
where $C>0$ is a universal constant.
\end{lemma}
\begin{proof}
We consider the projection of the random walk to the horizontal and
vertical axises, and denote them by $(X_t)$ and $(Y_t)$
respectively. Define
$$T_{_X} = \min\left\{t: |X_t - X_0| \geq \frac{n}{(\log n)^\kappa}\right\}\,, \mbox{ and } T_{_Y} = \min\{t: Y_t = 0\}\,.$$
It is clear that $\tau_{\partial A}\leq T_{_Y}$ and $T_{_X}
\leq \tau_{\partial \mathcal{B} \setminus \partial B}$. Write
$t^\star = \frac{n^2}{(\log n)^{3\kappa}}$. Is is obvious that with probability $1- \exp(-\Omega(t^\star))$ the number of steps spent on walking in the horizontal (vertical) axis is at least
$t^\star/3$. Combined with standard estimates
for 1-dimensional random walks, it follows that for a universal constant
$C>0$ (recall that $v$ has vertical coordinate bounded by $2n/(\log n)^{2\kappa}$)
$$\P(T_{_Y} \geq t^\star) \leq \frac{C}{(\log n)^{\kappa/2}}\,, \mbox{ and } \P(T_{_X} \leq t^\star) \leq \frac{C}{(\log n)^\kappa}\,.$$
Altogether, we conclude that $\P_v(\tau_{\partial A} >
\tau_{\partial \mathcal B \setminus \partial B}) \leq 2C (\log
n)^{-\kappa/2}$ as required.
\end{proof}

\noindent{\bf Thin points for random walks.}
In this subsection, we consider the embedded discrete time random walk, which is obtained by following the jumps in the continuous-time walk and ignore the exponentially distributed waiting times
between the jumps. For $v\in A$, denote by $N_v(t)$ the number of times that $v$ is visited in the embedded walk, which corresponds to the continuous-time walk up to $\tau(t)$. We aim to show the following
lemma on thin points of random walks, i.e., points that are visited for only a few times.
\begin{lemma}\label{lem-thin-points}
Write $t = m_L^2/2$ and assume $\kappa\geq 8$. Consider the random walk on $A$ up to $\tau(t)$. For $B\in \mathcal{B}$, define
$$R_B = \{\exists v\in B: N_v(t) \leq 120\}\,.$$
Then, $\P(R_B) \geq 10^{-6}$.
\end{lemma}
The following are two crucial ingredients for the verification of the preceding lemma.
\begin{prop}\cite[Prop. 4.10]{Ding11}\label{prop-thin-path}
Consider a network $G = (V, E)$ with maximal degree $\Delta$ and an arbitrary $t>0$. For $v\in V$, assume that $\{\ell_w: w\in V\}$ are
positive numbers such that $\ell_u\ell_v \leq 1/16$ for
all $u\sim v$. Define $\Gamma = \{L^w_{\tau(t)} = \ell_w \mbox{ for
all } w\in V\}$. Then
$$\P(N_v(t) \geq 30\Delta \mid \Gamma) \leq 1/2\,.$$
\end{prop}

\begin{lemma}\label{lem-detec-GFF}
Consider the Gaussian free field $\{\eta_v\}_{v\in A}$ with
Dirichlet boundary condition, and consider the collection of packing boxes
$\mathcal{B}$ defined as above. For $B\in \mathcal{B}$, define
$$E_B = \left\{\exists v\in B: |\eta_v - m_L|\cdot |\eta_u -  m_L| \leq 1/4 \mbox{ for all } u\sim v\right\}\,.$$
Then, $\P(E_B) \geq 10^{-4} / 16$.
\end{lemma}

\begin{proof}[\emph{\textbf{Proof of Lemma~\ref{lem-thin-points}}}]
Consider a $B\in \mathcal{B}$. Lemma~\ref{lem-detec-GFF} gives that
$\P(E_B) \geq 10^{-4}/16$. Combined with Theorem~\ref{thm:rayknight}, this implies that with probability at least $10^{-4}/16$, there exists $v \in B$ such that $L^v_{\tau(t)} \cdot L^u_{\tau(t)}\leq 1/64$ for all $u\sim v$. Combined with Proposition~\ref{prop-thin-path},
it follows that $\P(R_B) \geq 10^{-6}$.
\end{proof}

It remains to prove Lemma~\ref{lem-detec-GFF}. To this end, we will use a detection
result from \cite{Ding11}.
\begin{prop}\label{prop-gff-detection}\cite[Proposition 3.4]{Ding11}
Given a graph $G= (V, E)$ with maximal degree bounded by $\Delta$,
let $\{\eta_v\}_{v\in V}$ be the GFF on $G$ with $\eta_{v_0} = 0$
for some $v_0\in V$.  For any $M\geq
0$,
$$\P\left(\exists v\in V: |\eta_v - M|\cdot |\eta_u -  M| \leq 1/4 \mbox{ for all } u\sim v\right) \geq \frac{1}{4\cdot 10^{\Delta}} \P(\sup_u \eta_u \geq M)\,.$$
\end{prop}
Though Proposition~\ref{prop-gff-detection} was stated in \cite{Ding11} for the case that the GFF was defined with the value at a single vertex $v_0$ pinned at 0. The precisely identical proof shows that the same result holds if the values of the GFF are pinned to be 0 in an (arbitrary) set of vertices $U$ (that is, the covariances are given by Green functions as defined in \eqref{eq-def-green-function}). This slightly new version of the preceding proposition directly yields
$$\P\left\{\exists v\in B: |\eta_v - m^*(1/4)|\cdot |\eta_u -  m^*(1/4)| \leq 1/4 \mbox{ for all } u\sim v\right\} \geq 1/16\,,$$
where $m^*(1/4) = \sup\{z: \P(\sup_{v\in B} \eta_v \geq z) \geq 1/4\}$ is the $(1/4)$-quantile for $\sup_{v\in B} \eta_v$. It remains to prove that $m_L \geq m^*(1/4)$. This is an immediate consequence of the following lemma.
\begin{lemma}\label{lem-1/4-quantile}
For a graph $G = (V, E)$, consider $V_1\subset V_2 \subset V$. Let
$\{\eta^{(1)}_v\}_{v\in V}$ and $\{\eta^{(2)}_v\}_{v\in V}$ be GFFs
on $V$ such that $\eta^{(1)}|_{V_1} = 0$ and $\eta^{(2)}|_{V_2} =
0$ (that is to say, the covariances are given by Green functions as in \eqref{eq-def-green-function} with $U = V_1$ and $U = V_2$ respectively.). Then for any $\lambda \geq 0$ and $V'\subset V$,
$$\P(\mbox{$\sup_{v\in V'}$}\eta^{(1)}_v \geq \lambda) \geq \tfrac{1}{2} \P(\mbox{$\sup_{v\in V}$}\eta^{(2)}_v \geq \lambda)\,.$$
\end{lemma}
\begin{proof}
Note that the conditional covariance matrix of
$\{\eta^{(1)}_v\}_{v\in V'}$ given the values of
$\{\eta^{(1)}_v\}_{v\in V_2\setminus V_1}$ corresponds to the
covariance matrix of $\{\eta^{(2)}_v\}_{v\in V'}$. This implies that
$$\{\eta^{(1)}_v: v\in V'\}\stackrel{law}{=} \{\eta^{(2)}_v + \E(\eta^{(1)}_v \mid \{\eta^{(1)}_u: u\in V_2 \setminus V_1\}):  v\in V'\}\,,$$
where on the right hand side $\{\eta^{(2)}_v: v\in V'\}$ is
independent of $\{\eta^{(1)}_u: u\in V_2\setminus V_1\}$. Write
$\phi_v = \E(\eta^{(1)}_v \mid \{\eta^{(1)}_u: u\in V_2 \setminus
V_1\})$. Note that $\phi_v$ is a linear combination of
$\{\eta^{(1)}_u: u\in V_2 \setminus V_1\}$, and thus a mean zero
Gaussian variable. By the above identity in law, we derive that
$$\P(\mbox{$\sup_{v\in V'}$}\eta^{(1)}_v \geq \lambda)  = \P(\mbox{$\sup_{v\in V'}$} \eta^{(2)}_v + \phi_v \geq \lambda) \geq  \P(\mbox{$\sup_{v\in V'}$} \eta^{(2)}_v  \geq \lambda, \phi_{\xi} \geq 0) = \tfrac{1}{2}\P(\mbox{$\sup_{v\in V'}$} \eta^{(2)}_v  \geq \lambda)\,,$$
where we denote by $\xi \in V'$ the maximizer of $\{\eta^{(2)}_u: u\in
U\}$ and the second transition follows from the independence of
$\{\eta^{(1)}_v\}$ and $\{\phi_v\}$.
\end{proof}

\noindent {\bf Application of sprinkling method.} We are now ready to employ the sprinkling method and prove the lower bound on the cover time. As a preparation, we show that random walks on all the boxes $B\in \mathcal{B}$ are almost independent. To formalize the statement, we decompose the random walk up to $\tau(t)$ into a collection
$\mathcal{E}$ of disjoint excursions at the boundary $\partial A$, where an excursion is a minimal segment for the random walk such that the starting and ending vertex both belong to $\partial A$.
\begin{lemma}\label{lem-random-walk-independent}
For $t \leq 10 (\log n)^2$ and $\kappa \geq 8$, we have
$$\P(\exists \mathrm{Ex}\in \mathcal{E}\, \exists B, B'\in \mathcal{B}:  \mathrm{Ex}\cap B \neq \emptyset, \mathrm{Ex}\cap B'\neq \emptyset) = O(1/\log n)\,.$$
\end{lemma}
\begin{proof}
For $B\in \mathcal{B}$, denote by $\mathrm{ex}(B) = |\mathrm{Ex} \in \mathcal{E} : \mathrm{Ex}\cap B \neq\emptyset|$ the number of excursions that visited $B$. By our choice of the packing $\mathcal{B}$,
we trivially have that $R_{\mathrm{eff}}(B, \partial A) \geq 1$ for all $B\in \mathcal{B}$, and thus $\mathrm{ex}(B)$ is stochastically dominated by the sum of i.i.d.\ Bernoulli($1/|\partial A|$) where the number of summands
is an independent Poisson random variable of mean $t |\partial A|$. Therefore,
$$\P(\mathrm{ex}(B) \geq 20 (\log n)^2) = O(1/\log n)\,.$$
Combined with Lemma~\ref{lem-well-separated-packing}, we obtain that
$$\P(\exists \mathrm{Ex}\in \mathcal{E}\, \exists B'\in \mathcal{B}: \mathrm{Ex}\cap B \neq \emptyset, \mathrm{Ex}\cap B'\neq \emptyset) = O(1/(\log n)^{\kappa-2})\,.$$
Now a simple union bound over $B\in \mathcal{B}$ yields the desired estimate.
\end{proof}
Thanks to the preceding lemma, we can now assume without loss of generality that all the excursions visited at most one box $B\in \mathcal{B}$. Since the visits to all the boxes belong to
disjoint excursions, the independence among the excursions then implies the independence for random walks within all the boxes $B\in \mathcal{B}$.

 Write $t^- = (m_L - 1)^2/2$. For $B\in \mathcal{B}$, let $Q_B$ be the event that the box $B$ is not covered by time $\tau(t^-)$.
By Lemma~\ref{lem-thin-points}, we have $\P(R_B) \geq 10^{-6}$. Assume that $R_B$ indeed holds, and take $v\in B$ such that $N_v(t) \leq 120$, and thus $\mathcal{E}_v = \{\mathrm{Ex}\in \mathcal{E}: v\in \mathrm{Ex}\}$ has cardinality at most 120.
Given the collection $\mathcal{E}$ of excursions at the boundary $\partial A$ (note that here we do not yet reveal the order for excursions to occur), the times for the excursions to occur (measured by the local time
at boundary $\partial A$) are i.i.d.\ uniformly distributed over $[0, t]$. Therefore,
$$\P(Q_B) \geq \P(R_B)\P(\mathrm{Ex} \mbox{ occurs in } (t^-, t] \mbox{ for all } \mathrm{Ex}\in \mathcal{E}_v  \mid B, \mathcal{E}) \geq 10^{-6} (1/m_L)^{120}\,.$$
Recall from Theorem~\ref{thm-BZ} that
\begin{equation}\label{eq-M-L}
m_L = \sqrt{2/\pi} \big(\log n - \big(2\kappa + \tfrac{3}{8\log 2}\big)\log\log n\big) +  \tfrac{3 \kappa}{4\log 2}\log\log\log n + O(1)\,.
\end{equation}
 Now choose $\kappa = 400$, and we will have $m = |\mathcal{B}| \geq (\log n)^{130}/40$. By our justified assumption on the independence of events $\{Q_B: B\in \mathcal{B}\}$ and a standard argument on concentration, we conclude that
$$\P(\tau_{\mathrm{cov}} \leq \tau(t^-)) = O(1/\log n)\,.$$
Using \eqref{eq-M-L} again and applying Lemma~\ref{lem-concentration-tau-t}, we compete the proof for the lower bound.

\section{Random walk on 2D torus}

In this section, we consider the random walk on a 2D torus and prove Theorem~\ref{thm-free}. We wish to employ the same roadmap
as the proof for the wired boundary case, and we come across the following two conceptual difficulties:
\begin{enumerate}[(i)]
\item The inverse local time $\tau(t)$ is not concentrated enough. Indeed, it is not hard to see that $\tau(t)$ has a deviation of order $n^2 \sqrt{t} \sqrt{\log n} \gg n^2 \log n \log \log n$, for $t = \Theta((\log n)^2)$.
\item Essentially regardless of the choice of packing boxes, the random walk starting from one box would strongly prefer to hit other boxes before going back to the origin (since the origin, as a single point, is very likely to miss).
This reinforces the challenge to control the correlation between the random walks in different boxes.
\end{enumerate}

The hope (and the reason) that the deviation of the inverse local time does not give the right order for the deviation of the cover time, is that the number of excursions required to cover the
graph (or alternatively, the local time at the origin at time $\tau_{\mathrm{cov}}$) also exhibits a fairly large deviation (as opposed to the wired boundary case). Furthermore, these two random variables are negatively
associated such that their deviations will cancel each other, and hence the cover time $\tau_{\mathrm{cov}}$ is still fairly concentrated.

It seems rather challenging to study (or even to formalize) the negative association between the inverse local time and the number of excursions required for covering. Alternatively, we pose an artificial boundary over a suitable subset of the
torus and argue that:
\begin{itemize}
\item The cover times with or without this artificial boundary are almost the same.
\item The inverse local time exhibits a significantly smaller deviation with the boundary condition.
\item It is possible to carefully select a collection of packing boxes such that the random walks on these boxes are almost independent.
\end{itemize}
Our proof, as demonstrated in the rest of this section, is laid out precisely in this manner. Since we have been through the proof for the wired boundary case, we focus on the new issues for the case of 2D torus.
\subsection{A coupling of random walks}
Consider $\mathbb{Z}_n^2$ and $o \in \mathbb{Z}_n^2$. We imagine that $\mathbb{Z}_n^2$ is placed on a two-dimensional lattice where $o$ is the origin and the boundaries are properly identified.
Throughout this section, we use the notation
$$n_k = n/(\log n)^k \mbox{ for all } k\geq 0 \,.$$
For $r > 0$, define $\mathcal{C}_r$ to be the discrete ball of radius $r$, by
$$\mathcal{C}_r = \{x\in \mathbb{Z}^2: \|x\|_2 \leq r\}\,.$$
 Let $\kappa \geq 10$ be an absolute constant selected later. For convenience of notation, denote by $A = \mathbb{Z}_n^2$ in this section.
Let $\tilde{A}$ be obtained from $A$ by identifying all the vertices in $\mathcal{C}_{n_{2\kappa}}$ (as $v_0$). We consider a random walk $(S_t)$ on $A$ and a random walk $(\tilde{S}_t)$
on $\tilde{A}$, respectively. The following coupling says that with high probability these two random walks have the same behavior on $A \setminus \mathcal{C}_{n_\kappa}$.

\begin{lemma}\label{lem-coupling-random-walk}
Define $\tilde{\tau}(t)$ as in \eqref{eq:inverselt} to be the inverse local time for the random walk
$(\tilde{S}_r)$. For $t\leq 10 (\log n)^2$, with probability at least $1 - O(1/\log n)$, we can couple the random walk $(S_r)$ and $(\tilde S_r)$ together such that for a random time $\tau$ satisfying $|\tau - \tilde{\tau}(t)| \leq n^2/\log n$, we have the random walks
$(S_r: 0\leq r \leq \tau)$ and $(\tilde{S}_r: 0\leq r\leq \tilde\tau(t))$  are the same in the region $A\setminus \mathcal{C}_{n_{\kappa}}$.
\end{lemma}
\noindent{\bf Remarks.} (1) Note that in the coupling, we do not insist that the total amount of time spent at the two random walks are the same. All that we require is that if
we watch the two random walks in the region $A\setminus \mathcal{C}_{n_{\kappa}}$, we will observe the same sequence of random walk paths $(P_1, P_2, \ldots)$ where each $P_i$ is a random
walk path with starting and ending points in $\partial \mathcal{C}_{n_{\kappa}}$. (2) As we will see in the proof, the same result holds if we shift the disk $\mathcal{C}_{n_{2\kappa}}$ within distance $n/3$. We will use this fact in the derivation for the upper bound on the cover time.

In order to prove the preceding coupling lemma, we need to study the harmonic measure $H_{B}(x, \cdot)$ on $B$ (for $B\subset \mathbb{Z}^2$ and $x\in \mathbb{Z}^2$) defined by
$$H_B(x, y) = \P_x (S_{\tau_B} = y) \mbox{ for all } y\in B\,.$$
The following lemma (see, e.g., \cite[Prop. 6.4.5]{LL10}) will be used repeatedly.
\begin{lemma}\label{lem-harmonic-measure-coupling}
Suppose that $m<n/4$ and $\mathcal{C}_n \setminus \mathcal{C}_m \subset B\subset \mathcal{C}_n$. Suppose that $x\in \mathcal{C}_{2m}$ with $\P_x(S_{\tau_{\partial B}} \in \partial \mathcal{C}_n) >0$
and $z\in \partial \mathcal{C}_n$. Then,
$$\P_x(S_{\tau_{\partial B}} = z \mid S_{\tau_{\partial B}} \in \partial \mathcal{C}_n) = H_{\partial \mathcal{C}_n}(0, z)(1 + O(m \log(n/m))/n)\,.$$
Furthermore, we have that $c/n\leq H_{\partial \mathcal{C}_n}(0, z) \leq C/n$ for two absolute constants $c, C>0$.
\end{lemma}
For two probability measures $\mu$ and $\nu$ on a countable space $\Omega$, we define the total variation distance between $\mu$ and $\nu$ by
 $$\|\mu - \nu\|_{\mathrm{TV}} = \tfrac{1}{2}\sum_{x\in \Omega} |\mu(x) - \nu(x)|\,.$$
We will use a well-known fact that there exists a coupling $(X, Y)$ such that $X\sim \mu$, $Y\sim \nu$ and $\P(X\neq Y) = \|\mu - \nu\|_{\mathrm{TV}}$. See, e.g., \cite[Prop. 4.7]{LPW09}.
The following is an immediate consequence of Lemma~\ref{lem-harmonic-measure-coupling}.
\begin{cor}\label{cor-harmonic-coupling}
Denote by $H(\cdot, \cdot)$ and $\tilde{H}(\cdot, \cdot)$ the harmonic measures for random walks on $A$ and $\tilde{A}$ respectively. For all $x \in \mathcal{C}_{n_{2\kappa}}$, we have
$$\|H_{\partial \mathcal{C}_{n_{\kappa}}}(x, \cdot) - \tilde{H}_{\partial \mathcal{C}_{n_{\kappa}}}(v_0, \cdot)\|_{\mathrm{TV}} = O(1/(\log n)^8)\,.$$
\end{cor}
\begin{proof}
Let $B = \mathcal{C}_{n_{\kappa}}\setminus \mathcal{C}_{n_{2\kappa}}$ and $m = n_{2\kappa}$. It is clear that
\begin{align*}&\|H_{\partial \mathcal{C}_{n_{\kappa}}}(x, \cdot) - \tilde{H}_{\partial \mathcal{C}_{n_{\kappa}}}(v_0, \cdot)\|_{\mathrm{TV}}\\ \leq& \max_{y, z\in \partial \mathcal C_{2m}}\|\P_y(S_{\tau_{\partial B}} = \cdot \mid S_{\tau_{\partial B}} \in
\partial \mathcal C_{n_{\kappa}}) - \P_z(S_{\tau_{\partial B}} = \cdot \mid S_{\tau_{\partial B}} \in
\partial \mathcal C_{n_{\kappa}})\|_{\mathrm{TV}}\,.\end{align*}
Now, an application of Lemma~\ref{lem-harmonic-measure-coupling} with $n = n_{\kappa}$ completes the proof.
\end{proof}
\begin{proof}[\emph{\textbf{Proof of Lemma~\ref{lem-coupling-random-walk}}}]
In order to demonstrate the coupling, we consider the crossings between $\partial \mathcal{C}_{n_{2\kappa}}$ and $\partial \mathcal{C}_{n_{\kappa}}$ for $(S_r)$, where each crossing
is a minimal segment of the random walk path which starts at $\partial \mathcal{C}_{n_{2\kappa}}$ and ends at $\partial \mathcal{C}_{n_{\kappa}}$; for $(\tilde{S}_r)$, we consider the crossings
between $v_0$ and $\partial \mathcal{C}_{n_{\kappa}}$. We denote by $\tilde{K}$ the number of crossings for $(\tilde{S}_r)$ up to time $\tilde{\tau}(t)$, and denote by $(\tilde{Z}_{k})_{1\leq k\leq \tilde{K}}$ be the sequence of
ending points for these crossings. Similar to the justifications of \eqref{eq-fact}, we see that $\tilde{K}$ is distributed as a sum of i.i.d.\ Bernoulli variables with mean $\frac{1}{\mathrm{deg}_{v_0} R_{\mathrm{eff}}(v_0, \partial \mathcal{C}_{n_{\kappa}})}$ and the number of summands
is an independent Poisson variable with mean $\mathrm{deg}_{v_0} t$, where $\mathrm{deg}_{v_0}$ is the degree of the identified vertex $v_0$ in $\tilde{A}$. Since $R_{\mathrm{eff}}(v_0, \partial \mathcal{C}_{n_{\kappa}}) \geq 1$, we have
$$\P(\tilde{K} \geq (\log n)^3) = O(1/\log n).$$
In what follows, we can then assume that $\tilde{K} \leq (\log n)^3$. Now, we consider the first $\tilde{K}$ crossings for random walk $(S_r)$ and denote by $(Z_k)_{1\leq k\leq \tilde{K}}$ the sequence
of the ending points for these crossings. Observe that
$$\|\P(Z_k \in \cdot) - \P(\tilde{Z}_k \in \cdot)\|_{\mathrm{TV}} \leq \max_{x\in \mathcal{C}_{n_{2\kappa}}}\|H_{\partial \mathcal{C}_{n_{\kappa}}}(x, \cdot) - \tilde{H}_{\partial \mathcal{C}_{n_{\kappa}}}(v_0, \cdot)\|_{\mathrm{TV}} = O(1/(\log n)^8)\,.$$
Therefore, with probability at least $1 - O(1/(\log n)^5)$, we have $Z_k = \tilde{Z}_k$ for all $1\leq k\leq \tilde{K}$. In what follows, we assume that we indeed have $Z_k = \tilde{Z}_k$.

Now, the coupling is natural and obvious. Since starting from the same point at $\partial \mathcal{C}_{n_{\kappa}}$, the random walks on $A$ and $\tilde{A}$ follow the same transition kernel until they
hit $\partial \mathcal{C}_{n_{2\kappa}}$. Thus, we can couple the two random walks together such that the sequences of the random walk paths watched in the region $A\setminus \mathcal{C}_{n_{\kappa}}$ are identical to each other.

It remains to control the difference between $\tau$ and $\tilde{\tau}(t)$. Due to the coupling, we see that the total time that these two random walks spent on the region $A \setminus \mathcal{C}_{n_{2\kappa}}$ are the same. So the difference only comes from
the time the two walks spend at $\mathcal{C}_{n_{\kappa}}$.  We denote by $T$ and $\tilde{T}$ these two times respectively. Note that
$$\E T \leq (\log n)^3  \max_{x\in \mathcal{C}_{n_{2\kappa}}}\E_x \tau_{\partial \mathcal{C}_{n_{\kappa}}} \leq (\log n)^3 O((n_{\kappa})^2 \log n) = O(n^2/(\log n)^{16})\,.$$
Since $\kappa\geq 10$, we have
$$\E \tilde{T} =\sum_{v\in \mathcal{C}_{n_{\kappa}}} d_v t \leq \sum_{v\in \mathcal{C}_{n_{10}}} d_v t  = O(n^2/(\log n)^{18})\,.$$
At this point, an application of Markov's inequality completes the proof.
\end{proof}

Thanks to the coupling, it suffices to study the covering for random walk $(\tilde{S}_r)$ in order to understand the cover time for $(S_r)$. In what follows,
we will focus on the random walk $(\tilde{S}_r)$. For easiness of notation, we drop the tilde symbol except for the underlying graph $\tilde{A}$ (to remind us which graph we are working on in case of ambiguity). One of the purposes to identify $\mathcal{C}_{n_{2\kappa}}$ is to give better concentration for the inverse local time, for which we first prove the next preparation lemma.
\begin{lemma}\label{lem-green-function-2}
Let $G_{\cdot}(\cdot, \cdot)$ be the Green function of random walk on $\mathbb{Z}^2$. For $x , y \in A$,
$$G_{y}(x, x) = \tfrac{4}{\pi}(\log |x - y|_{\mathbb{Z}_n^2}) + O(1)\,,$$
where $|x-y|_{\mathbb{Z}_n^2} = \min_{i, j\in \mathbb{Z}}|x - y + (in, jn)|$ is the Euclidean distance between $x,y$ in $\mathbb{Z}_n^2$.
\end{lemma}
\begin{proof}
By homogeneity of $\mathbb{Z}_n^2$, we can assume that $x= o$ and $y_1 \geq y_2\geq 0$. Let $L$ be a vertical line segment of length $y_1$ centered at $y/2$, and let $L' = \{v\in L: \P_v(\tau_o < \tau_x)\geq 1/2\}$. Without loss of generality, we assume that $|L'| \geq y_1/2$ (otherwise we exchange the role of $o$ and $y$). By our assumption, we have $G_{y}(o, o) \leq 2G_{L'}(o, o)$. In addition, we have
\begin{equation}\label{eq-G-L'}
G_{L'}(o, o) = G_{\partial \mathcal C_{y_1/4}}(o, o) + \max_{z\in \mathcal C_{y_1/4}} G_{y}(z, o) = \tfrac{2}{\pi} \log y_1 + \max_{z\in \mathcal C_{y_1/4}} G_{L'}(z, o) + O(1)\,,\end{equation}
where in the second equality we used Lemma~\ref{lem-green-function}. Let $B$ be a rectangular centered at $o$ with side lengths $y_1\times 2y_1$. So in particular $L\subset \partial B$. Applying \cite[Proposition 6.4.3]{LL10}, we obtain that the random walk started from $z$ will hit $\partial B$ before returning to $o$ with probability $1- O(1/\log y_1)$ for all $z\in \mathcal{C}_{y_1/4}$. Also note that the harmonic measure of $L'_1$ with respect to starting point $z$ and stopping set $\partial B$ is at least $c>0$ for a certain constant $c$ (see, e.g., \cite[Prop. 8.1.5]{LL10}, and note that $L'$ consists of a constant fraction of $\partial B$; alternatively, one could approximate the harmonic measure by that of the Brownian motion.) Therefore, we have deduced that $\P_z(\tau_o < \tau_{L'}) = O(1/\log y_1)$. In addition, we could get $G_{L'}(o,o) \leq G_{\partial \mathcal C_{y_1/4}}(o, o)/c = O(\log y_1)$. Altogether, we get that $\max_{z\in \mathcal C_{y_1/4}} G_{L'}(z, o) = O(1)$. Combined with \eqref{eq-G-L'}, this completes the proof of the upper bound.

In order to prove the lower bound, we use a connection between Green functions and effective resistances as follows (see \cite[Thm. 9.20]{Janson97})
\begin{equation}\label{eq-Janson}
G_{U}(x, y) = 2 (R_{\mathrm{eff}}(x, U) + R_{\mathrm{eff}}(y, U) - R_{\mathrm{eff}}(x, y))\,.
\end{equation}
Now, let $B_1$ and $B_2$ be cubes of side-length $y_1/4$ centered at $o$ and $y$ respectively. Using preceding inequality and Lemma~\ref{lem-green-function}, we see that
$$R_{\mathrm{eff}}(o, \partial B_1) = R_{\mathrm{eff}}(o, \partial B_2) = \tfrac{1}{2\pi} \log y_1 + O(1)\,.$$
Using Rayleigh monotonicity for the effective resistances (c.f., e.g., \cite{LP}), we deduce that
$$R_{\mathrm{eff}}(o, y_1) \geq \tfrac{1}{\pi} \log y_1 + O(1)\,.$$
Combined with \eqref{eq-Janson}, this yields the desired lower bound.
\end{proof}

The following lemma (whose lower bound we did not attempt to optimize) will be useful.
\begin{lemma}\label{lem-green-function-3}
For $m<n/4$, we have $\frac{1}{8}\log (n/m) + O(1) \leq R_{\mathrm{eff}}(\partial \mathcal{C}_m, \partial \mathcal{C}_n) = \frac{1}{2\pi} \log (n/m) + O(1)$.
\end{lemma}
\begin{proof}
By Rayleigh monotonicity principle, we have  $R_{\mathrm{eff}}(o, \partial \mathcal{C}_n) \geq R_{\mathrm{eff}}(\partial \mathcal{C}_m, \partial \mathcal{C}_n) + R_{\mathrm{eff}}(o, \partial \mathcal{C}_m)$.  Combined with Lemma~\ref{lem-green-function}, the upper bound follows. In order to prove the (non-optimal) lower bound, it suffices to consider the cut-set $\Pi_k$ where $\Pi_k$ are the edges connecting centering boxes of side length $2k$ and $2(k+1)$. Applying Nash-William (c.f. \cite[Ch. 2]{LP}) criterion, we obtain that $$R_{\mathrm{eff}}(\partial \mathcal{C}_m, \partial \mathcal{C}_n)\geq \mbox{$\sum_{k=m}^{n/4}$} (|\Pi_k|)^{-1} = \tfrac{1}{8} \log (n/m) + O(1)\,,$$ completing the proof.
\end{proof}

The next corollary is immediate.
\begin{cor}\label{cor-green-function}
For all $x, y\in A \setminus \mathcal{C}_{n_\kappa}$, we have
\begin{align}
\tfrac{2}{\pi}\log |x|+ \tfrac{1}{2} \log \tfrac{|x|}{n_{2\kappa}}  &\leq G_{ \mathcal{C}_{n_{2\kappa}}}(x, x)  \leq \tfrac{2}{\pi}\log |x|+ \tfrac{2}{\pi} \log \tfrac{|x|}{n_{2\kappa}} + O(1)\,, \label{eq-cor-1}\\
G_{ \mathcal{C}_{n_{2\kappa}}}(x, y)& = \tfrac{2}{\pi} (\log n - \log |x-y|_{\mathbb{Z}_2} + O(\kappa \log\log n)) \,. \label{eq-cor-2}
\end{align}
\end{cor}
\begin{proof}
By Rayleigh monotonicity, we have $R_{\mathrm{eff}}(x, o) \geq R_{\mathrm{eff}}(x, \partial \mathcal{C}_{n_{2\kappa}}) + R_{\mathrm{eff}}(o, \partial \mathcal{C}_{n_{2\kappa}})$. Combined with Lemma~\ref{lem-green-function}, it follows that
$$R_{\mathrm{eff}}(x, \partial \mathcal{C}_{n_{2\kappa}}) \leq \tfrac{1}{2\pi}\log |x| + \tfrac{1}{2\pi} \log \tfrac{|x|}{n_{2\kappa}} + O(1)\,.$$
Using Rayleigh monotonicity again, we get
$$R_{\mathrm{eff}}(x, \partial \mathcal{C}_{n_{2\kappa}}) \geq R_{\mathrm{eff}}(x, \partial \mathcal C_x(|x|/3)) + R_{\mathrm{eff}}(\mathcal{C}_{n_{2\kappa}}, \partial \mathcal{C}_{|x|/3}) \geq \tfrac{1}{2\pi}\log |x|+ \tfrac{1}{8} \log \tfrac{|x|}{n_{2\kappa}} + O(1)\,,$$
where $\mathcal{C}_x(|x|/3)$ is a ball of radius $|x|/3$ centered at $x$ and $\mathcal{C}_{|x|/3}$ is a ball of radius $|x|/3$ centered at the origin. Therefore,
we obtain \eqref{eq-cor-1} in view of \eqref{eq-Janson}.
 An analogous bound holds for $R_{\mathrm{eff}}(y, \partial \mathcal{C}_{n_2\kappa})$. Writing $r = |x-y|_{\mathbb{Z}^2} \wedge n_{2\kappa}$, we have
$$R_{\mathrm{eff}}^{\tilde A}(x, y) \geq R_{\mathrm{eff}}(x, \partial \mathcal{C}_x(r/4)) + R_{\mathrm{eff}}(y, \partial \mathcal{C}_y(r/4)) = \tfrac{1}{\pi} \log r + O(1) = \tfrac{1}{\pi} \log |x-y|_{\mathbb{Z}_n^2} + O(\kappa \log \log n)\,.$$
By Lemma \ref{lem-green-function} and Rayleigh monotonicity, we see from Lemma~\ref{lem-green-function-2} that
$$R_{\mathrm{eff}}^{\tilde A}(x, y) \leq R_{\mathrm{eff}}^A(x, y) = \tfrac{1}{\pi} \log |x-y|_{\mathbb{Z}_n^2} + O(1)\,.$$
Thus, $R_{\mathrm{eff}}^{\tilde A}(x, y) = \tfrac{1}{\pi} \log |x-y|_{\mathbb{Z}_n^2} + O(\kappa \log\log n)$. Combined with \eqref{eq-cor-1} and \eqref{eq-Janson}, \eqref{eq-cor-2} follows.
\end{proof}

\begin{lemma}\label{lem-concentration-tau-t-tilde}
For $t>0$, define $\tau(t)$ as in
\eqref{eq:inverselt} to be the inverse local time for the random walk on $\tilde{A}$. Then, for any $\lambda > 0$
$$\P(|\tau(t) - 2t|E|| \geq \lambda n^2\sqrt{\log \log n}\sqrt{t }) \leq O(\kappa/\lambda^2)\,.$$
\end{lemma}
\begin{proof}
Our proof follows the same outline as that of Lemma~\ref{lem-concentration-tau-t}. We only emphasize the different estimates required due to the change of the
underlying graph. As in the proof of Lemma~\ref{lem-concentration-tau-t}, we consider the GFF $\{\eta_x: x\in \tilde{A}\}$ on $\tilde{A}$ (that is, the covariances are given by Green functions as in \eqref{eq-def-green-function} with $U = \mathcal{C}_{n_{2\kappa}}$), and define $Z_x = \eta_x^2 - \E\eta_x^2$ and $Z = \sum_x Z_x$. Applying the Lemma~\ref{lem-green-function-2}, we obtain that
$$\var Z \leq \sum_{x, y\in \tilde A\setminus \mathcal{C}_{n_\kappa}}(G_{\mathcal{C}_{n_{2\kappa}}}(x, y))^2 \leq n^2 O(1)\sum_{k =1}^n k(\log n - \log k + \kappa\log \log n)^2 = O(\kappa n^4 \log \log n)\,.$$
Furthermore, we have
$$\var(\mbox{$\sum_x$} \eta_x) = \sum_x \E_x \tau_{\partial \mathcal{C}_{n_{2\kappa}}} = O(\kappa n^4 \log\log n)\,.$$
Using the above two estimates and following the proof of Lemma~\ref{lem-concentration-tau-t}, we can easily deduce the standard deviation for the time spent by random walk on $\tilde A\setminus \mathcal{C}_{n_\kappa}$ up to time $\tau(t)$ is $O(n^2\sqrt{ \kappa \log \log n}\sqrt{t })$. It remains to control the deviation for the time spent on $\mathcal{C}_{n_\kappa}$, for which a simple Markov's inequality suffices (as the volume of $\mathcal{C}_{n_\kappa}$ is significantly smaller than $n^2$ and thus the time spent by the random walk on it is negligible compared to $\tau(t)$). Altogether, the proof is completed.
\end{proof}

\subsection{Proof of Theorem~\ref{thm-free}}
We first explain the proof for the upper bound on the cover time for random walk on $\tilde{A}$, based on which we derive an upper bound for the random walk on $A$.
By Corollary~\ref{cor-green-function}, we get that for all $x\in \tilde{A}$,
$$G_{\mathcal{C}_{n_{2\kappa}}}(x,x) \leq \tfrac{2}{\pi}\log n + O(\kappa) \log\log n \,.$$
Applying this estimate and following the proof in Subsection~\ref{sec:upper-dirichlet}, we can derive that
$$\P(\tau_{\mathrm{cov}}(\tilde{A}) \geq \tau(t_\lambda)) = O(1/\lambda^2)\,,$$
for $t_\lambda = \frac{1}{\pi}(\log n + C\kappa \log \log n + \lambda)^2$ with a large enough absolute constant $C>0$. Combined with Lemma~\ref{lem-concentration-tau-t-tilde}, we obtain that
$$\P(\tau_{\mathrm{cov}}(\tilde{A}) \geq 2t_\lambda |E| + \lambda \sqrt{t_\lambda} |E| \sqrt{\log\log n}) = O(\kappa/\lambda^2)\,.$$
Now, applying Lemma~\ref{lem-coupling-random-walk} twice (with $\mathcal{C}_{2\kappa}$ centered at $o$ and $(n/3, n/3)$ respectively), we conclude that
$$\P(\tau_{\mathrm{cov}}(A) \geq 2t_\lambda |E| + \lambda \sqrt{t_\lambda} |E| \sqrt{\log\log n}) = O(\kappa/\lambda^2)\,,$$
completing the proof for the upper bound.

\bigskip

The proof for the lower bound is more involved. To this end,
we specify a packing of balls in $\tilde{A}$. Throughout this subsection, we denote by $v_0$ the vertex obtained from
identifying $\mathcal{C}_{n_{2\kappa}}$. Let $m = (\log n)^{\kappa/2}/2$, and define $\mathcal{B}$ to be a collection of packing balls $\{B_i: i\in [m]\}$ such that:
\begin{itemize}
\item Take $\{o_i\in \partial {C}_{n_{\kappa/2}} : i\in [m]\}$ such that $|o_i - o_j| \geq n_{\kappa}$ for all $1\leq i< j\leq m$.
\item For $i\in [m]$, the packing ball $B_i$ is a disk of radius $n_{5\kappa}$ centered at $o_i$.
\end{itemize}
We show that the packing balls in $\mathcal{B}$ are well-separated in the sense that the random walks watched in each $B\in \mathcal{B}$ are
almost independent.
\begin{lemma}\label{lem-well-separated-packing-2}
Fix $t\leq (\log n)^2$ and set $t' = t - \kappa^2 \log n \log \log n$. For all $B\in \mathcal{B}$, let $\mathcal{L}_B$ be the law for the path of random walk $(S_r)$ watched in the region $B$ up to $\tau(t)$. Let $\{P_B: B\in \mathcal{B}\}$ be independent random walk paths such that
$P_B \sim \mathcal{L}_B$. Let $\{P^*_B: B\in \mathcal{B}\}$ be the random walk paths in $B\in \mathcal{B}$ that are generated
 by  $(S_r)$ up to $\tau(t')$. Then for $\kappa \geq 1000$, we can construct a coupling such that with probability at least
$1- O(1/\log n)$, we have $P^*_B \subseteq P_B$  for all $B\in \mathcal{B}$, i.e., $\cup_{\gamma\in P^*_B} \gamma \subseteq \cup_{\gamma\in P_B} \gamma$.
\end{lemma}

Note that the random walk in ball $B$ can influence the random walk in ball $B'$ only by either influencing the number of times for the random walk to enter $B'$ or the
hitting vertex at which the random walk enters $B'$. In order to prove the preceding lemma, we only need to control this two types of correlations. In what follows, we formalize this intuition.

For $B\in \mathcal{B}$, let $U_B = \{v_0\} \cup \mathcal{B} \setminus B$. We consider the crossings from $U_B$ to $B$, that is,
a minimal segment of the random walk with starting point in $U_B$ and ending point in $B$. To be formal, we define $\tau_0 = \tau'_0 = 0$ and for all $k\in \N$
$$\tau_k = \min\{r > \tau'_{k-1} : S_r \in B\}\,, \mbox{ and } \tau'_k = \min\{r > \tau_k: S_r \in U_B\}\,.$$
Then the set of crossing points up to time $\tau(s)$ are defined by
$$\mathcal{N}_B(s) = \{S_{\tau'_k}: k\geq 1, \tau'_k \leq \tau(s)\}\,.$$
In view of Lemma~\ref{lem-well-separated-packing-2}, define
$$\mathcal{N}^*_B = \mathcal{N}_B(t')\,, \mbox{  for } B\subset \mathcal{B}\,.$$
In addition, let $\{\mathcal{N}_B\}_{B\in \mathcal{B}}$ be independent such that $\mathcal{N}_B$ has the same law as $\mathcal{N}_B(t)$. To prove Lemma~\ref{lem-well-separated-packing-2},
it suffices to prove that there exists a coupling such that with probability at least $1- O(1/\log n)$
\begin{equation}\label{eq-to-show}
\mathcal{N}^*_B \subseteq \mathcal{N}_B\,, \mbox{ for all } B\in \mathcal{B}\,.
\end{equation}
Naturally, a preliminary step is to show that $|\mathcal{N}^*_B|$ is smaller than  $|\mathcal{N}_B|$, as incorporated in the following lemma.
\begin{lemma}\label{lem-number-of-crossings}
Suppose that $\kappa\geq 1000$. Then with probability at least $1 - O(1/\log n)$
$$|\mathcal{N}^*_B| \leq |\mathcal{N}_B|\leq (\log n)^2 \mbox{ for all } B\in \mathcal{B} \,.$$
\end{lemma}
\begin{proof}
Consider $B\in \mathcal{B}$ and let $N = N_B = |\mathcal{N}_B|$. Following the justifications for \eqref{eq-fact}, we can write
\begin{equation}\label{eq-representation}
N = \sum_{k=1}^K X_k\,,\end{equation}
where $K \sim \mathrm{Poi}(t/R_{\mathrm{eff}}(v_0, B))$ is the number of excursions at $v_0$ that enters $B$, and $X_k$ is the
number of crossings to $B$ in the $k$-th excursion that visited $B$ and thus clearly independent. In addition, it is not hard to see that
$$1 + \mathrm{Geom}(p_1) \preceq X_k \preceq 1 + \mathrm{Geom}(p_2)\,,$$
where we use the definition that $\P(\mathrm{Geom}(p) = k) = p^k (1-p)$ for $k\geq 0$, and that
\begin{align*}p_1 &= \min_{x\in B}\P_x(\mbox{ there is at least one crossing from $U_B$ to $B$ before hitting } v_0),\\
 p_2 &= \max_{x\in B}\P_x(\mbox{ there is at least one crossing from $U_B$ to $B$ before hitting } v_0)\,.
\end{align*}
Let $\mathcal{C}$ be a discrete ball of radius $n_{2\kappa}$ with the same center as $B$. By Lemma~\ref{lem-harmonic-measure-coupling}, we have
$$p_1 - p_2 \leq \max_{x, y \in B} |H_{\partial \mathcal{C}}(x, \cdot) - H_{\partial \mathcal{C}}(y, \cdot)|_{\mathrm{TV}} = O((\log n)^{-2\kappa + 1})\,.$$
Next, we give an upper bound on $p_2$. To this end, we apply Corollary~\ref{cor-green-function} and obtain that for all $x\in U_B$
$$G_{v_0}(x, x) \geq \tfrac{1}{2\pi}(\log |x| + \tfrac{\pi}{4}\log \tfrac{|x|}{n_{2\kappa}}) + O(1)\,, \mbox{ and } G_{B}(x, x) \leq \tfrac{1}{2\pi}(\log |x-o'| + \log \tfrac{|x-o'|}{n_{5\kappa}}) +O(1)\,,$$
where $o'$ is the center of $B$. It follows that for $n$ large enough, we have
\begin{equation}\label{eq-compare-res}
R_{\mathrm{eff}}(x, v_0) \geq R_{\mathrm{eff}}(x, B)\,.\end{equation}
Based on preceding inequality, we now claim that
\begin{equation}\label{eq-claim-prob}
\P_x(\tau_B < \tau_{v_0}) \leq 1/2\,.\end{equation}
In order to verify the claim, we consider a reduced network (for network reduction, see, e.g., \cite[Lemma 2.9]{DLP10}) on vertices $\{x, v_0, b\}$ where $b$ is for the vertex obtained by identifying $B$ in the original graph. Let $c_{b, v_0}$, $c_{v_0, x}$ and $c_{o, b}$ be the conductance between pairs of vertices in the reduced network. Since the effective resistances are preserved by network reduction, we deduce from \eqref{eq-compare-res} that $c_{v_0 , x} \leq c_{v_0, b}$. Combined with the fact (c.f. \cite[Lemma 2.9]{DLP10}) that the original random walk watched on $\{v_0, b, x\}$ has the same law as the random walk on the reduced network, \eqref{eq-claim-prob} follows since in the reduced network the random walk (started at $v_0$) has probability at most $1/2$ moving to $x$ instead of $b$.
By definition of $p_2$, it is clear that
$$p_2 \leq \max_{x\in U_B}\P_x(\tau_B < \tau_{v_0}) \leq 1/2\,.$$
Applying Rayleigh monotonicity principle again, we obtain that
$$R_{\mathrm{eff}}(v_0, B) \leq R_{\mathrm{eff}}(o, o') - R_{\mathrm{eff}}(o, \partial \mathcal{C}_{n_{2\kappa}}) - R_{\mathrm{eff}}(o', \partial B) \leq 10\kappa \log\log n\,,$$
where in the last transition we used Lemmas~\ref{lem-green-function} and \ref{lem-green-function-2}. Denoting by $\alpha = R_{\mathrm{eff}}(v_0, B)/\kappa\log\log n$, we therefore have  $\alpha\leq 10$. Combining all the estimates for the parameters to determine the distribution of $N$, we see that
$$\E N - \E |\mathcal{N}^*_B| \geq \tfrac{\kappa \log n}{\alpha}\,,  \var N\leq \tfrac{4(\log n)^2}{\alpha\kappa \log\log n}\,,  \mbox{ and } \var |\mathcal{N}^*_B| \leq \tfrac{4(\log n)^2}{\alpha\kappa \log\log n}\,.$$
At this point, a standard concentration argument yields that (recalling \eqref{eq-representation})
$$\P(N < |\mathcal{N}^*_B|) \leq \P(|N - \E N| \geq \tfrac{\kappa \log n}{2\alpha}) + \P(||\mathcal{N}^*_B| - \E |\mathcal{N^*_B}|| \geq \tfrac{\kappa \log n}{2\alpha}) \leq O(1) (\log n)^{-\kappa/16 \alpha}\,.$$
Combined with the fact that $\alpha\leq 10$ and $|\mathcal{B}| = (\log n)^{\kappa/2}$ (recalling the assumption that $\kappa \geq 1000$), the desired estimates follows from a simple union bound.
\end{proof}

Next, we prove that both $\mathcal{N}^*_B$ and $\mathcal{N}_B$ are almost a collection of i.i.d.\ random points from $B$.
To this end, we need to control the harmonic measure on $\partial B$ when the random walk is started at a vertex outside of $B$, as incorporated in the next lemma. For a proof,
see \cite[Prop. 6.6.1]{LL10} and its proof therein.
\begin{lemma}\label{lem-harmonic-measure-outside}
For $m< n$ and $x, y\in \mathbb{Z}^2 \setminus \mathcal{C}_n$, we have that
$$\|H_{\mathcal{C}_m}(x, \cdot) - H_{\mathcal{C}_m}(y, \cdot)\|_{\mathrm{TV}} = O\big(\tfrac{m (\log n)^2}{n}\big)\,.$$
\end{lemma}
The following is an immediate consequence.
\begin{cor}\label{cor-harmonic-coupling-2}
For a random walk started at an arbitrary $x\in U_B$, let $\mathcal{P}$ be the random walk watched in the region $U_B$. Let $P_1$ and $P_2$ be two arbitrary traces
in region $U_B$ before hitting $B$. Then,
$$\|\P_x(S_{\tau_B} \in \cdot \mid \mathcal{P} = P_1) - \P_x(S_{\tau_B} \in \cdot \mid \mathcal{P} = P_2)\|_{\mathrm{TV}} = O((\log n)^{-3\kappa+2})\,.$$
\end{cor}
\begin{proof}
Let $\mathcal{C}$ and $\mathcal{C}'$ be discrete balls of radius $n_{2\kappa}$ and $n_{3\kappa/2}$ with the same center as $B$. Clearly, the random walk needs to hit $\mathcal{C}$ before hitting $B$.
Consider $y, z\in \partial \mathcal{C}$. We have the following decomposition for $y$ (and thus similarly for $z$)
\begin{equation}\label{eq-measure-decompose}
\P_y(S_{\tau_B} = u) = \P_y(S_{\tau_B = u} \mid \tau_B < \tau_{\mathcal C}') \P_y(\tau_B < \tau_{\mathcal C}') + \P_y(\tau_B > \tau_{\mathcal C}')\sum_{v\in \partial \mathcal{C}'}\mu_y(v)\P_v(S_{\tau_B} = u)\,,\end{equation}
where $\mu_y(\cdot)$ is a certain probability distribution (whose particular form is of no care). We also need the following lemma.
\begin{lemma}\label{lem-annulus}\cite[Prop. 6.4.1]{LL10}
For $m< n$ and $x\in \mathcal{C}_n \setminus \mathcal{C}_m$, we have
$$\P_x(\tau_{\partial \mathcal{C}_n} < \tau_{\partial \mathcal{C}_m}) = \frac{\log |x| - \log m + O(1/m)}{\log n - \log m}\,.$$
\end{lemma}
\noindent Applying the preceding lemma, we obtain that
\begin{equation}\label{eq-measure-decompose-2}
 \P_y(\tau_B < \tau_{\partial \mathcal{C}'}) \geq 1/100\,,  \mbox{ and } |\P_y(\tau_B < \tau_{\partial \mathcal{C}'}) - \P_z(\tau_B < \tau_{\partial \mathcal{C}'})| = O(1/\sqrt{n})\,.\end{equation}
Combining \eqref{eq-measure-decompose} and \eqref{eq-measure-decompose-2} with an application of Lemma~\ref{lem-harmonic-measure-outside}, we obtain that
$$\|\P_y(S_{\tau_B \in \cdot} \mid \tau_B < \tau_{\mathcal C'}) - \P_z(S_{\tau_B \in \cdot} \mid \tau_B < \tau_{\mathcal C'})\|_{\mathrm{TV}} = O((\log n)^{-3\kappa+2})\,.$$
Combined with \eqref{eq-measure-decompose-2}, this immediately implies the desired estimates.
\end{proof}

\begin{proof}[\emph{\textbf{Proof of Lemma~\ref{lem-well-separated-packing-2}}}]
By Lemmas~\ref{lem-number-of-crossings} and \ref{cor-harmonic-coupling-2}, a union bound would imply that with probability at least $1- O(1/\log n)$, we have $\mathcal{N}^*_B\subseteq \mathcal{N}_B$
for all $B\in \mathcal{B}$. Together with the discussion around \eqref{eq-to-show}, this completes the proof.
\end{proof}

The lower bound for the cover time now readily follows. For each $B\in \mathcal{B}$ consider a box $B'\subset B$ of side length $L = n_{5\kappa}/4$. Let $m_L$ be the median for the GFF on a $L \times L$ box with
Dirichlet boundary (recall Theorem~\ref{thm-BZ} for its estimate). For $D\subset \tilde{A}$, denote by $\tau_{\mathrm{cov}}(D)$ the first time that the set $D$ is covered. Let $Q_{B'}$ be the event that the box $B$ is not covered by random walk path $P_B$ defined as in Lemma~\ref{lem-well-separated-packing-2}.
By Lemma~\ref{lem-well-separated-packing-2}, $Q_{B'}$ are \emph{independent} events such that with probability at least $1-O(1/\log n)$ for all such $B'$
$$Q_{B'} \subseteq \{\tau_{\mathrm{cov}}(B') \geq \tau((m_L - 2 \kappa^2 \log\log n -1)^2/2)\}\,,$$
where
$$\P(Q_{B'}) = \P(\tau_{\mathrm{cov}}(B') \geq \tau((m_L - 1)^2/2))\,.$$
We need to relate $m_L$ to $\tilde M = \max_{v\in B} \tilde{\eta}_v$ where $\tilde{\eta}_\cdot$ is the GFF on 2D torus with Green functions given by \eqref{eq-def-green-function} with $U = \{v_0\}$.
In light of the Lemma~\ref{lem-1/4-quantile}, we see that $\P(\tilde M \geq m_L) \geq 1/4$.
Following the proof for the wired boundary case (first show a vertex with small local times at it and its neighbors, and then do a sprinkling argument), we can show that
$$\P(Q_{B'}) \geq 10^{-6} (m_L)^{-120}\,.$$
Choosing $\kappa = 1000$ and applying Lemma~\ref{lem-well-separated-packing-2}, we obtain that
$$\P(\tau_{\mathcal{B}} \geq \tau((m_L- 2 \kappa^2 \log\log n - 1)^2/2)) = 1 - O(1/\log n)\,.$$
Combined with Lemmas~\ref{lem-coupling-random-walk} and \ref{lem-concentration-tau-t-tilde}, this completes the proof for the lower bound.

\subsection*{Acknowledgements}

We are most grateful to Ofer Zeitouni for numerous interesting discussions, and we thank two anonymous referees for very detailed and useful comments on an earlier manuscript.

\end{document}